\documentclass[english,11pt]{article}
\usepackage{tikz}
\usetikzlibrary{decorations.pathreplacing}
\usepackage{enumerate}
\usepackage[shortlabels]{enumitem}
\usepackage{verbatim}
\usepackage[margin=1in]{geometry}
\usepackage{amssymb}
\usepackage{caption}
\usepackage{bbm}
\usepackage{amsthm}
\usepackage{enumitem}

\usepackage{scalerel}    %%%%  used first for scaling \bullet  
\usepackage{stmaryrd}   %%%% for integer interval notation 

\makeatletter
\DeclareRobustCommand\widecheck[1]{{\mathpalette\@widecheck{#1}}}
\def\@widecheck#1#2{%
    \setbox\z@\hbox{\m@th$#1#2$}%
    \setbox\tw@\hbox{\m@th$#1%
       \widehat{%
          \vrule\@width\z@\@height\ht\z@
          \vrule\@height\z@\@width\wd\z@}$}%
    \dp\tw@-\ht\z@
    \@tempdima\ht\z@ \advance\@tempdima2\ht\tw@ \divide\@tempdima\thr@@
    \setbox\tw@\hbox{%
       \raise\@tempdima\hbox{\scalebox{1}[-1]{\lower\@tempdima\box
\tw@}}}%
    {\ooalign{\box\tw@ \cr \box\z@}}}
\makeatother

\usepackage{rotating,centernot,cancel}    %% to manipulate symbols 

\usepackage{amsmath}
\usepackage{tocloft}
\usepackage{float}

\usepackage{tcolorbox}   
\PassOptionsToPackage{hyphens}{url}\usepackage{hyperref}
\usepackage{babel}
\theoremstyle{plain}

\input amssym.def
\input amssym.tex

\def\beq{\begin{equation}}
\def\eeq{\end{equation}}
\def\beqn{\begin{eqnarray}}
\def\eeqn{\end{eqnarray}}

\usepackage{mathtools}
\DeclarePairedDelimiter\floor{\lfloor}{\rfloor}

\newtheorem{theorem}{Theorem}[section]
\newtheorem{corollary}[theorem]{Corollary} 
\newtheorem{lemma}[theorem]{Lemma} 
\newtheorem*{lemma*}{Lemma}
\newtheorem{proposition}[theorem]{Proposition} 

\theoremstyle{remark}

\theoremstyle{definition}

\numberwithin{figure}{section}

\def\R{\mathbb R}

\def\P{\mathbb P}

\def\to{\rightarrow}
\def\dlim[#1][#2]{\lim_{#1 \to #2, #1 \neq #2}}

\def\Mid{\textup{Mid}}

\newcommand{\be}{\begin{equation}}
\newcommand{\ee}{\end{equation}}
\newcommand\mydots{\hbox to 1em{.\hss.\hss.}}

\newcommand{\myfootnote}[1]{
    \renewcommand{\thefootnote}{}
    \footnotetext{\scriptsize#1}
    \renewcommand{\thefootnote}{\arabic{footnote}}
}

\allowdisplaybreaks

\title{Large deviations of geodesic midpoint fluctuations in last-passage percolation with general i.i.d.~weights}

\author{Tom Alberts\thanks{\scriptsize{Department of Mathematics, University of Utah. \texttt{alberts@math.utah.edu}}}\qquad
Riddhipratim Basu\thanks{\scriptsize{International Centre for Theoretical Sciences, 
Tata Institute of Fundamental Research. \texttt{rbasu@icts.res.in}}}\qquad
Sean Groathouse\thanks{\scriptsize{Department of Mathematics, University of Utah. \texttt{sean@math.utah.edu}}}\qquad 
Xiao Shen\thanks{\scriptsize{Department of Mathematics, University of Utah. \texttt{xiao.shen@utah.edu}}}}
\date{}
\begin{document}
\maketitle

\begin{abstract}
The study of transversal fluctuations of the optimal path is a crucial aspect of the Kardar-Parisi-Zhang (KPZ) universality class. In this work, we establish the large deviation limit for the midpoint transversal fluctuations in a general last-passage percolation (LPP) model with mild assumption on the i.i.d.~weights. The rate function is expressed in terms of the right tail large deviation rate function of the last-passage value and the shape function.
When the weights are chosen to be i.i.d.~exponential random variables, our result verifies a conjecture communicated to us by Liu \cite{Liu-22a}, showing the asymptotic probability of the geodesic from $(0,0)$ to $(n,n)$ following the corner path $(0,0) \to (n,0) \to (n,n)$ is $({4}/{e^2})^{n+o(n)}$.

\end{abstract}

\myfootnote{Date: \today}
\myfootnote{2010 Mathematics Subject Classification. 60K35, 	60K37}
\myfootnote{Key words: large deviation, geodesics, last-passage percolation, Kardar-Parisi-Zhang.}
%\tableofcontents

\section{Introduction}

%The concept of universality lies at the heart of probability and statistical physics, describing how systems with distinct microscopic details can exhibit the same universal macroscopic behavior. A classical example is the \textit{central limit theorem} (CLT), which asserts that the large-scale behavior of a scaled and centered sum becomes independent of the specific distribution of the individual summands, converging to the Gaussian distribution as the number of summands tends to infinity.

%In 1986, Kardar, Parisi, and Zhang \cite{Kar-Par-Zha-86} predicted a fundamentally different universality class, now known as the \textit{KPZ universality class}. This class describes a new type of universal behavior expected to arise in a wide range of stochastic models with spatial dependence. Over the years, extensive computer simulations and laboratory experiments have demonstrated that the KPZ universality class is remarkably diverse, encompassing percolation models, directed polymers, interacting particle systems, random tilings, stochastic partial differential equations, and more.

Last-passage percolation (LPP) is one of many stochastic models characterized by random growth along optimal paths in a random environment. A central focus in LPP, and more broadly in the Kardar-Parisi-Zhang (KPZ) universality class \cite{Kar-Par-Zha-86}, is the study of \textit{transversal fluctuations}. These fluctuations quantify the deviation of the optimal path from reference lines and play a critical role in uncovering the geometric properties and space-time growth profiles of the model. 

The main focus of this paper is a large deviation principle for unusually large transversal fluctuations of the optimal path, on the scale of the distance between the two endpoints of the geodesic. The KPZ universality class is broadly divided into two categories: \textit{exactly solvable models}, where rigorous mathematical results can be derived using intricate algebraic methods, and \textit{general} or \textit{non-solvable models}, for which such fine tools are not directly applicable. Our approach is probabilistic and applies to general models beyond the exactly solvable cases. The main results are presented in the next section.

\subsection{Main results}
Last-passage percolation is a well-studied model in the KPZ universality class, yet has seen limited work on large deviation properties of the associated random geometry. In the LPP model, independent and identically distributed (i.i.d.)\ weights $\{\omega_{\mathbf{z}}\}_{\mathbf{z} \in \mathbb{Z}^2}$ are associated with the integer lattice $\mathbb{Z}^2$. Given two coordinatewise-ordered points ${\mathbf{u}}, {\mathbf{v}}$ of $ \mathbb{Z}^2$, an up-right path between them is a nearest neighbor path $\gamma = \{\gamma_i\}_{i=0}^{|\mathbf{u} - \mathbf{v}|_1}$ such that 
$$
\gamma_0 = \mathbf{u}, \quad \gamma_{|\mathbf{u} - \mathbf{v}|_1} = \mathbf{v}, \textup{ and }\gamma_{i+1} - \gamma_i \in \{\mathbf{e}_1, \mathbf{e}_2\},
$$ 
where $\mathbf{e}_1$, $\mathbf{e}_2$ are the standard basis vectors of $\R^2$. The weight of each up-right path is the sum of the $\omega$ variables along the path, and the \textit{last-passage value} $G_{\mathbf{u}, \mathbf{v}}$ is then defined to be the maximum weight among all up-right paths from $\mathbf{u}$ to $\mathbf{v}$ while excluding the starting point $\mathbf{u}$, i.e. 
\begin{equation}\label{G}
G_{\mathbf{u}, \mathbf{v}} =  \max_{\substack{\textup{up-right}\\{ \gamma: \mathbf u \to \mathbf v}}} \;\;\sum_{\mathbf{z} \in  \gamma\setminus\{\mathbf{u}\}} \omega_{\mathbf{z}}.
\end{equation}
We set $G_{\mathbf{u}, \mathbf{v}} = -\infty$ if there is no up-right path between $\mathbf{u}$ and $\mathbf{v}$. 
The maximizing paths for $G_{\mathbf{u}, \mathbf{v}}$ are commonly referred to as the \textit{geodesics}. %Lastly, we note that in our definition of \( G_{\mathbf{u}, \mathbf{v}} \), the weight at the starting point \( \mathbf{u} \) is omitted. Including this weight does not alter the geodesics.

In this paper, we will assume that our i.i.d.~weights $\{\omega_\mathbf{z}\}$ are non-negative and satisfy 
\begin{align}\label{wa}
\begin{cases}
\exists \,\alpha > 0 \textup{ such that } \mathbb{E}\big[e^{\alpha\omega_{\mathbf{z}}}\big] < \infty,\\
\omega_{\mathbf{z}} \textup{ is continuous, i.e. its CDF is continuous}\\
\forall\, \beta < \infty, \mathbb{P}(\omega_{\mathbf{z}} > \beta) > 0.
\end{cases}
\end{align}
Among the three assumptions stated, the exponential moment condition is essential to our analysis. The other two assumptions, while not strictly necessary, serve to simplify notation and streamline our arguments. Specifically, the continuity assumption guarantees the uniqueness of geodesics; however, this could be adjusted by considering the rightmost geodesics in cases where uniqueness does not hold. Additionally, the unbounded support of the weights ensures that the large deviation rate function for the right tail of the last-passage value remains finite.

Throughout the paper, we will assume that $n$ is an even integer. This is only for the simplification of notation, as all of our results and proofs also hold when $n$ is odd.  Let $\Mid_{0,n}$ denote the midpoint of the geodesic from $(0,0)$ to $(n,n)$, which we regard as a random point on the line $$\mathcal{L}_{n/2} = \{ (x,y) \in \R^2 : x + y = n \}.$$  Our main result below gives the large deviation estimate for the point $\Mid_{0,n}$. To present this result, we first introduce the limiting shape function of the last-passage values and their right tail large deviations rate function.

We mostly use the passage times $G_{\mathbf{0}, \mathbf{v}}$ where $\mathbf{v}$ is on the line $x+y = n$, which we typically write as $\mathbf{v} = (\floor{n/2 + tn}, \floor{n/2 - tn})$ with $-\frac12 \leq t \leq \frac12$. The right tail large deviation rate function of the last-passage value is defined as
$$
{\mathcal{J}}_t(r) = -\lim_{n\to \infty} \frac{1}{n}\log \mathbb{P}(G_{\mathbf{0},(\floor{n/2+tn}, \floor{n/2-tn})} \geq rn).
$$
The limit exists due to Fekete's Subadditive Lemma. Section \ref{lpp_rate} has more information and background on this rate function. We will also make heavy use of the shape function 
\[
\mu_t = \lim_{n\to\infty} \frac{G_{\mathbf{0}, (\floor{n/2+tn}, \floor{n/2-tn})}}{n}.
\]
The existence of the limit is guaranteed by Kingman's Subadditive Ergodic Theorem \cite{Dur-96}.

\begin{theorem}\label{main_thm}
Fix $0< t \leq 1/2$ such that $\mu_0 > \mu_t$. For each $\epsilon>0$, there exists a constant $n_0$ such that for each $n \geq n_0$, it holds that 
$$
e^{-2n(\mathcal{J}_t(\mu_0) + \epsilon)}  \leq \mathbb{P}\Big(\Mid_{0,n} \cdot \mathbf{e}_1 =  \floor{n/2+tn}\Big) \leq  \mathbb{P}\Big(\Mid_{0,n} \cdot \mathbf{e}_1 \geq  \floor{n/2+tn}\Big) \leq e^{-2n(\mathcal{J}_t(\mu_0) - \epsilon)}.
$$
Thus, in the limit, we have
$$
2{\mathcal{J}}_t(\mu_0) = -\lim_{n\to \infty} \frac{1}{n} \log \mathbb{P}\Big(\Mid_{0,n} \cdot \mathbf{e}_1 =  \floor{n/2+tn}\Big) = -\lim_{n\to \infty} \frac{1}{n} \log \mathbb{P}\Big(\Mid_{0,n} \cdot \mathbf{e}_1 \geq  \floor{n/2+tn}\Big).
$$
\end{theorem}

The assumption $\mu_0 > \mu_t$ is equivalent to requiring $(n/2 + tn, n/2- tn)$ to be outside of the diagonal cone for which the limit shape is flat. As long as the weights are not almost surely constant, the flat stretch cannot take up the entire interval $t \in [0, 1/2]$, i.e.\ there is an $\eta \in (0, 1/2)$ such that $\mu_0> \mu_t$ for all $t \in [\eta, 1/2]$, which we now explain. First, note we have a strictly inequality $\mu_0 > \mu_{1/2}$ because 
$$
\mu_0   = \lim_{n\to \infty} \frac{1}{n} G_{\mathbf{0}, (n/2,n/2)} \geq \lim_{n\to \infty} \frac{1}{n}\sum_{i=1}^{n/2} \Big[\omega_{(i, i)} + \max\{\omega_{(i+1, i)}, \omega_{(i, i+1)} \} \Big] >\lim_{n\to \infty} \frac{1}{n}\sum_{i=1}^{n} \omega_{(i, 0)} =  \mu_{1/2}.
$$ 
Then, it is known that $\mu_t$ is continuous in $t$ up to the boundary \cite[Theorem 2.3]{Mar-04}, so our claim follows.
In addition, this assumption is natural because, in the only known scenario where this flat segment appears in the limit shape---namely, due to the supercritical oriented percolation at \(\operatorname{esssup} \omega_{\mathbf{z}}\), which is conjectured in \cite[Question 2.5.4]{Auf-Dam-Han-17} to be the only cause of this phenomenon---we have \( \mathcal{J} = \infty \) for directions strictly inside this cone. This is illustrated in Figure \ref{fig2}.

\begin{figure}[t]
\begin{center}

\tikzset{every picture/.style={line width=0.75pt}} %set default line width to 0.75pt        

\begin{tikzpicture}[x=0.75pt,y=0.75pt,yscale=-1.2,xscale=1.2]

%Straight Lines [id:da5178555377410599] 
\draw [color={rgb, 255:red, 155; green, 155; blue, 155 }  ,draw opacity=1 ] [dash pattern={on 0.84pt off 2.51pt}]  (120.52,189.82) -- (241.31,148.2) ;
\draw [shift={(243.2,147.55)}, rotate = 160.99] [color={rgb, 255:red, 155; green, 155; blue, 155 }  ,draw opacity=1 ][line width=0.75]    (10.93,-3.29) .. controls (6.95,-1.4) and (3.31,-0.3) .. (0,0) .. controls (3.31,0.3) and (6.95,1.4) .. (10.93,3.29)   ;
%Straight Lines [id:da6159890866968656] 
\draw [color={rgb, 255:red, 155; green, 155; blue, 155 }  ,draw opacity=1 ] [dash pattern={on 0.84pt off 2.51pt}]  (120.52,189.82) -- (161.05,70.94) ;
\draw [shift={(161.7,69.05)}, rotate = 108.83] [color={rgb, 255:red, 155; green, 155; blue, 155 }  ,draw opacity=1 ][line width=0.75]    (10.93,-3.29) .. controls (6.95,-1.4) and (3.31,-0.3) .. (0,0) .. controls (3.31,0.3) and (6.95,1.4) .. (10.93,3.29)   ;
%Straight Lines [id:da8376758808845242] 
\draw [color={rgb, 255:red, 155; green, 155; blue, 155 }  ,draw opacity=1 ] [dash pattern={on 0.84pt off 2.51pt}]  (240.02,70.32) -- (106.1,114.92) ;
\draw [shift={(104.2,115.55)}, rotate = 341.58] [color={rgb, 255:red, 155; green, 155; blue, 155 }  ,draw opacity=1 ][line width=0.75]    (10.93,-3.29) .. controls (6.95,-1.4) and (3.31,-0.3) .. (0,0) .. controls (3.31,0.3) and (6.95,1.4) .. (10.93,3.29)   ;
%Straight Lines [id:da35382271633761064] 
\draw [color={rgb, 255:red, 155; green, 155; blue, 155 }  ,draw opacity=1 ] [dash pattern={on 0.84pt off 2.51pt}]  (240.02,70.32) -- (196.35,197.66) ;
\draw [shift={(195.7,199.55)}, rotate = 288.93] [color={rgb, 255:red, 155; green, 155; blue, 155 }  ,draw opacity=1 ][line width=0.75]    (10.93,-3.29) .. controls (6.95,-1.4) and (3.31,-0.3) .. (0,0) .. controls (3.31,0.3) and (6.95,1.4) .. (10.93,3.29)   ;
%Straight Lines [id:da27293643779493926] 
\draw [color={rgb, 255:red, 245; green, 166; blue, 35 }  ,draw opacity=1 ]   (120.52,189.82) -- (275.53,87.15) ;
\draw [shift={(277.2,86.05)}, rotate = 146.48] [color={rgb, 255:red, 245; green, 166; blue, 35 }  ,draw opacity=1 ][line width=0.75]    (10.93,-3.29) .. controls (6.95,-1.4) and (3.31,-0.3) .. (0,0) .. controls (3.31,0.3) and (6.95,1.4) .. (10.93,3.29)   ;
%Shape: Circle [id:dp23976616713339927] 
\draw  [fill={rgb, 255:red, 0; green, 0; blue, 0 }  ,fill opacity=1 ] (235.83,70.32) .. controls (235.83,68.01) and (237.71,66.13) .. (240.02,66.13) .. controls (242.33,66.13) and (244.2,68.01) .. (244.2,70.32) .. controls (244.2,72.63) and (242.33,74.5) .. (240.02,74.5) .. controls (237.71,74.5) and (235.83,72.63) .. (235.83,70.32) -- cycle ;
%Shape: Circle [id:dp333599895638091] 
\draw  [fill={rgb, 255:red, 0; green, 0; blue, 0 }  ,fill opacity=1 ] (116.33,189.82) .. controls (116.33,187.51) and (118.21,185.63) .. (120.52,185.63) .. controls (122.83,185.63) and (124.7,187.51) .. (124.7,189.82) .. controls (124.7,192.13) and (122.83,194) .. (120.52,194) .. controls (118.21,194) and (116.33,192.13) .. (116.33,189.82) -- cycle ;
%Straight Lines [id:da02105804074122819] 
\draw [color={rgb, 255:red, 155; green, 155; blue, 155 }  ,draw opacity=1 ]   (130.7,79.9) -- (240.2,190.55) ;
%Shape: Free Drawing [id:dp1596622681596378] 
\draw  [line width=1.5] [line join = round][line cap = round] (123.2,188.05) .. controls (127.12,187.27) and (132.44,183.47) .. (135.2,181.05) .. controls (137.69,178.87) and (141.5,170.54) .. (144.2,170.05) .. controls (152.27,168.58) and (157.95,167.61) .. (164.2,166.05) .. controls (167.75,165.16) and (170.15,158.43) .. (173.7,158.05) .. controls (178.18,157.57) and (182.71,157.88) .. (187.2,157.55) .. controls (189.57,157.37) and (192.39,152.04) .. (194.2,149.55) .. controls (195.32,148.01) and (198.65,149.63) .. (199.7,148.05) .. controls (200.81,146.38) and (199.71,144) .. (200.2,142.05) .. controls (200.55,140.64) and (211.18,125.73) .. (211.2,125.55) .. controls (211.84,120.92) and (211.41,116.16) .. (212.2,111.55) .. controls (212.52,109.72) and (218.11,108.95) .. (219.7,107.05) .. controls (222.94,103.16) and (224.74,88.24) .. (228.2,83.05) .. controls (228.99,81.87) and (230.98,82.28) .. (232.2,81.55) .. controls (236.84,78.76) and (238.86,74.25) .. (240.2,69.55) ;

% Text Node
\draw (101.7,197.8) node [anchor=north west][inner sep=0.75pt]    {$(0,0)$};
% Text Node
\draw (224.7,45.8) node [anchor=north west][inner sep=0.75pt]    {$(n,n)$};
% Text Node
\draw (238,186.4) node [anchor=north west][inner sep=0.75pt]    {$\mathcal{L}_{n/2}$};

\end{tikzpicture}

\captionsetup{width=.8\linewidth}
\caption{Looking within the supercritical oriented percolation cone at \((0,0)\) in the northeast direction and the cone at \((n,n)\) in the southwest direction, if we fix a direction (shown in orange) that lies strictly inside the cone at $(0,0)$, as $n$ becomes large, with high probability, there will exist a geodesic formed by concatenating open paths from within the two cones, whose midpoint is to the right of the chosen direction.} \label{fig2}
\end{center}
\end{figure}
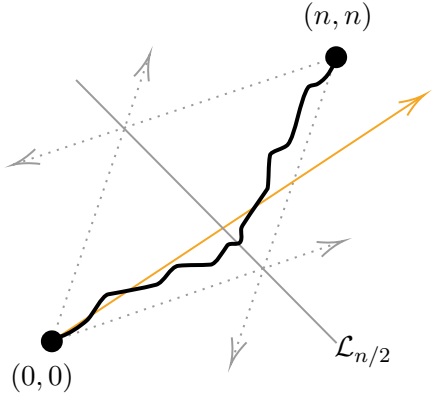

Next, we record a similar result for the endpoint of the point-to-line geodesic. The point-to-line last-passage value is defined to be 
$$
{G}_{\mathbf{0}, \mathcal{L}_{n/2}} = \max_{k\in \mathbb{Z}} {G}_{\mathbf{0}, (\floor{n/2+k}, \floor{n/2-k})}
$$
Let $\textup{End}_{0,n}$ denote the endpoint of the point-to-line geodesic.
\begin{proposition}
Fix $0< t \leq 1/2$ such that $\mu_0 > \mu_t$. For each $\epsilon>0$, there exists a $n_0$ such that for each $n \geq n_0$, it holds that 
$$
e^{-n(\mathcal{J}_t(\mu_0) + \epsilon)} \leq  \mathbb{P}\Big(\textup{End}_{0,n} \cdot \mathbf{e}_1 =  \floor{n/2 +tn}\Big) \leq  \mathbb{P}\Big(\textup{End}_{0,n} \cdot \mathbf{e}_1\geq  \floor{n/2 +tn}\Big) \leq e^{-n(\mathcal{J}_t(\mu_0) - \epsilon)}.
$$
Thus, in the limit, it holds that 
$$
{\mathcal{J}}_t(\mu_0) = -\lim_{n\to \infty} \frac{1}{n} \log \mathbb{P} \Big(\textup{End}_{0,n} \cdot \mathbf{e}_1 = \floor{n/2 + tn}\Big) = -\lim_{n\to \infty} \frac{1}{n} \log \mathbb{P}\Big(\textup{End}_{0,n} \cdot \mathbf{e}_1\geq \floor{n/2 + tn} \Big). 
$$
\end{proposition}

\subsection{Relations to the existing literature}

Precise large deviation principles (LDP) for the transversal fluctuations have not been derived in the past. For general weight distributions, the only result has been a concentration inequality in a closely related model known as \textit{first-passage percolation}, see \cite[Lemma 9.10]{Kes-86-stflour}. To state this result in the LPP setting, let $D_{0, n}$ denote the displacement of the geodesic $\Gamma_{0, n}$ from the diagonal line $y=x$. This result states that when the shape function has an exposed point in the diagonal direction (i.e. the limit shape has no flat segment at the diagonal), then
$$
\lim_{\epsilon\to 0} \liminf_{\\n\to\infty} \mathbb{P}(D_{0,n} < \epsilon n) =  1.
$$
Even for exactly solvable models, there has been little prior work on the large deviations principle of the transversal fluctuations for exactly solvable lattice models. A recent paper of Liu \cite{Liu-22a} provides an exact formula for the probability of a geodesic passing through a specified edge in the exactly solvable exponential LPP. Liu's formula is expressed in terms of iterated contour integrals and, as noted by the author, it is challenging to extract asymptotics from it.

For the directed landscape, another exactly solvable model which is known to be the scaling limit of certain exactly solvable models and is conjectured to be the universal scaling limit for all models within the KPZ universality class \cite{KPZ_DL}, the LDP for the last-passage value and the transversal fluctuations of geodesics have been established in \cite{dl_ldp}. However, since the directed landscape arises as a scaling limit of LPP under KPZ scaling, the transversal large deviations in the directed landscape correspond to \textit{moderate} deviations in pre-limit models. For the moderate deviations of the transversal fluctuation in FPP, under the curvature assumptions of the limit shape, the works \cite{MR3189069, MR3010809, transfluc} established that $D_{0,n} \leq n^{3/4 + o(1)}$ with high probability. In exponential LPP a series of papers demonstrated the exponential decay of the probability of the rare event that $D_{0,n}$ exceeds the typical scale $n^{2/3}$ \cite{aga-23, timecorrflat, slowbondproblem, ham-sar-20, lb_tf}. More recently, the precise leading-order constant in the exponential decay of the probabilities for large transversal fluctuations was determined in \cite{midpoint}. Our approach is partly inspired by the geometric arguments presented in that work.

Lastly, regarding the rate function \(\mathcal{J}_t\), a variational formulation over paths was recently derived in \cite{var_J_fpp} in the context of FPP, resembling the large deviation rate function established for the directed landscape in \cite{dl_ldp}.

\subsection{The $t=1/2$ Case}

The $t=1/2$ case corresponds to $\Mid_{0,n} = (n,0)$, or in other words the geodesic traverses the extreme corner path $(0,0) \to (n,0) \to (n,n)$. In private communication Liu used heuristics from his exact formula in \cite{Liu-22a} to conjecture an asymptotic formula for $\log \P (\Mid_{0,n} = (n,0))$ in the exponential LPP setting. We verify his conjecture using Theorem \ref{main_thm}. 

\begin{corollary}\label{corner}
Assume the $\omega_{\mathbf{z}}$ are i.i.d.~exponential random variables with rate one. Let $\Gamma_{0,n}$ denote the geodesic between $(0,0)$ and $(n,n)$. Then
\[
-\lim_{n\to \infty} \frac{1}{n} \log \mathbb{P} \bigg( (n,0) \in \Gamma_{0,n} \bigg) = 2 - 2 \log 2.
\]
\end{corollary}

The application of Theorem \ref{main_thm} is simple. First, since there is only one path from $(0,0) \to (n,0)$ it follows that $J_{1/2}$ is the rate function of a sum of i.i.d.~random variables. This property of $J_{1/2}$ holds for general weights $\omega_{\mathbf{z}}$ satisfying our assumptions. For rate one exponential random variables it is known that $J_{1/2}(x) = x - \log x - 1$, and the shape function satisfies $\mu_0 = 2$ thanks to \cite{Ros-81}. Thus Corollary \ref{corner} follows.

It is interesting to compare the asymptotics of Corollary \ref{corner} with those of a uniformly chosen up-right path $\mathbf{U}_{0,n}$ from $(0,0) \to (n,n)$, for which
\[
-\lim_{n\to \infty} \frac{1}{n} \log \mathbb{P} \bigg( (n,0) \in \mathbf{U}_{0,n} \bigg) = 2 \log 2 > 2 - 2 \log 2.
\]
In other words, the exponential LPP geodesic has an exponentially larger probability of following the extreme corner path than a uniformly chosen walk does. One intuition behind this phenomenon, as explained in \cite{Alb-Cat-21}, is that the corner path shares weights with a comparatively smaller number of paths, and therefore faces less competition to being the geodesic. The same intuition should be correct for any large deviation of the transversal fluctuations, leading us to conjecture that
\[
2{\mathcal{J}}_{t}(\mu_0) \leq - \lim_{n \to \infty} \frac{1}{n} \log \P \bigg( \Mid(\mathbf{U}_{0,n}) \cdot \mathbf{e}_1 = \floor{n/2+tn} \bigg)
\]
for all $t \in (0, 1/2]$. Note this inequality should hold for a broad class of weight distributions. In this case,  setting \( t = 1/2 \) makes both \( \mathcal{J}_{1/2} \) and the right hand side limit relatively explicit. This raises the possibility of deriving meaningful bounds on \( \mu_0 \) for general weight distributions.

\subsection*{Acknowledgments} The authors thank Zhipeng Liu for his insightful discussions on his work related to this problem.
TA is partially supported by NSF Grant DMS-1811087 and a Simons Foundation Grant. RB is partially supported by a MATRICS grant (MTR/2021/000093) from SERB, Govt.\ of India, DAE project no.\ RTI4001 via ICTS, and the Infosys Foundation via the Infosys-Chandrasekharan Virtual Centre for Random Geometry of TIFR. XS is partially supported by the Wylie
Research Fund at the University of Utah.

\section{Preliminaries}
%\subsection{Notation}

To simplify the notation, let us extend the LPP model onto $\mathbb{R}^2 \times \mathbb{R}^2$. Without switching the notation, let
$$G_{(a,b), (c,d)} = G_{(\floor a, \floor b), (\floor c,\floor d)},$$
and we note that this change does not affect the asymptotic results: suppose $J$ is a positive constant, for any large constant $r$
$$ e^{-\floor{r}(J+ \epsilon)} \leq e^{-r(J+ \frac\epsilon2)} \leq e^{-\floor{r}(J+ \frac\epsilon2)} \quad \textup{ and } \quad e^{-\floor{r}(J - \epsilon)} \leq e^{-r(J - 2\epsilon)} \leq e^{-\floor{r}(J - 2\epsilon)}.$$

In the following sections, we briefly recall the existing literature on the right tail large deviations of the last-passage value.

\subsection{Right tail large deviation rate function for the last-passage value}\label{lpp_rate}

In this section, we record a few properties of the right tail large deviation rate function of the last-passage value. These results are well established and follow from the superadditivity of the LPP model. Similar results have appeared in \cite{Kes-86-stflour} for FPP and \cite{LDP_poly} for the directed polymer model. Here we note that right-tail deviations in FPP correspond to left-tail deviations in LPP, and conversely, left-tail deviations in FPP correspond to right-tail deviations in LPP. 

Recall the shape function $\mu_t = \lim_{n\to\infty} n^{-1} G_{(0,0), (n/2+tn, n/2-tn)}$.

\begin{proposition}[{\cite[Theorem 3.2]{LDP_poly} \cite[Theorem 5.2]{Kes-86-stflour}}] \label{propJ}
For each $-1/2 \leq t \leq 1/2$ and $r >0$, the following $\mathbb{R}_{\geq 0}$-valued limit exists:
$${\mathcal{J}}_t(r) = -\lim_{n\to \infty} \frac{1}{n}\log \mathbb{P}(G_{\mathbf{0},(n/2+tn, n/2-tn)} \geq rn).$$
The function $\mathcal{J}_t(\bullet)$ is continuous and convex. Moreover, $\mathcal{J}_t(r) = 0$ if and only if $r \leq \mu_t$, and it is strictly increasing for $r \geq \mu_t$.
\end{proposition}

Next, we will record several known results involving the direction parameter $t$, beginning with the convexity of \( {\mathcal{J}}_t(r) \) in its two variables \( (t, r) \).
\begin{proposition}[{\cite[Theorem 3.3]{LDP_poly}}]
The rate function ${\mathcal{J}}_t(r)$ is  convex in its variables $(t,r)$. More precisely, fix $r_1, r_2 \in \mathbb{R}_{>0}$, $-1/2 \leq t_1 \leq t_2 \leq 1/2$ and $\lambda \in (0,1)$. Let $(t,r) = \lambda (t_1, r_1) + (1-\lambda) (t_2, r_2)$. Then,
$${\mathcal{J}}_t(r) \leq \lambda {\mathcal{J}}_{t_1}(r_1) + (1-\lambda) {\mathcal{J}}_{t_2}(r_2).$$
Furthermore, the function  ${\mathcal{J}}_t(r): (0, 1/2)\times \mathbb{R}_{>0} \to \mathbb{R}_{\geq 0}$ is continuous.
\end{proposition}

A second consequence is a monotonicity property in the directional parameter $t$.
\begin{proposition}\label{mono_t}
For each $r \in \mathbb{R}_{>0}$ and $0\leq t_1 < t_2 \leq 1/2$, it holds that   $${\mathcal{J}}_{t_1}(r) \leq {\mathcal{J}}_{t_2}(r).$$
\end{proposition}
\begin{proof}
First, we will show that ${\mathcal{J}}_{0}(r) \leq {\mathcal{J}}_{t}(r)$ for all $0<t\leq 1/2$. This directly follows from our convexity result. Rewrite $0 = -t/2 + t/2$, by convexity and symmetry
$${\mathcal{J}}_0(r) \leq \frac{1}{2} {\mathcal{J}}_{-t}(r) + \frac{1}{2} {\mathcal{J}}_{t}(r) = {\mathcal{J}}_{t}(r).$$
Now fix $0\leq t_1 < t_2 \leq 1/2$, let us rewrite $t_1 = \lambda t_2 + (1-\lambda)0$ where $\lambda = \frac{t_1}{t_2}$. By convexity again, 
$${\mathcal{J}}_{t_1}(r) \leq \lambda {\mathcal{J}}_{t_2}(r) + (1-\lambda) {\mathcal{J}}_{0}(r) \leq \lambda {\mathcal{J}}_{t_2}(r) + (1-\lambda) {\mathcal{J}}_{t_2}(r) = {\mathcal{J}}_{t_2}(r),$$
and this completes the proof.
\end{proof}

\subsection{Large deviation estimates of the last-passage value}

In this section, we record the following estimates of the last-passage value on the large deviation scale. 
The right tail large deviation estimate is stated below. 
\begin{proposition}\label{ld_right}
For each $|t| \leq 1/2$ and $\epsilon>0$, there exists a positive constant $n_0$ such that for each $n\geq n_0$ and $x>0$, the following holds
$$
e^{-n(\mathcal{J}_t(\mu_t + x) + \epsilon)} \leq \mathbb{P}(G_{\mathbf{0}, (n/2 + tn,n/2-tn)} \geq \mu_t n + x n) \leq e^{-n\mathcal{J}_t(\mu_t + x)}.
$$
\end{proposition}
The lower bound follows from Proposition \ref{propJ}, while the upper bound follows from Fekete's Subadditive Lemma 
$$
\mathcal{J}_t(r) = \lim_{n\to \infty} -\frac{1}{n} \log \mathbb{P}(G_{\mathbf{0}, (n/2 + tn,n/2-tn)} \geq rn) = \inf_n \bigg\{ -\frac{1}{n} \log \mathbb{P}(G_{\mathbf{0}, (n/2 + tn,n/2-tn)} \geq rn) \bigg\}
$$
as $\{-\log \mathbb{P}(G_{\mathbf{0}, (n/2 + tn,n/2-tn)} \geq rn)\}_n$ is a subadditive sequence of real numbers.

The left tail large deviation estimate is stated below.

\begin{proposition}\label{ld_left}
For each $\epsilon>0$, there exists a positive constant $c$ such that 
$$\mathbb{P}(G_{\mathbf{0}, (n/2 ,n/2)} \leq \mu_0 n - \epsilon n) \leq e^{-c n^2}.$$
\end{proposition}

The above proposition is essentially the same as \cite[Theorem 5.2]{Kes-86-stflour}, which proves the upper tail large deviations in FPP with bounded edge weights. It was discussed in \cite[Section 4.1]{BGS19} how the same proof can be adapted to the setting of Proposition \ref{ld_left} but this was not formally stated as a result. For the sake of completeness, we shall provide a short sketch of the proof.

\begin{proof}[Sketch]
The first step is to notice that for fixed but large $K$, on the event $\{G_{\mathbf{0}, (n/2 ,n/2)} \leq \mu_0 n - \epsilon n\}$, one has $G_{(3iK,0), (n/2, n/2-3iK)}\le \mu_0 n - \epsilon n$ for each $0\le i\le \delta n$. Let us denote $G^{K}_{(3iK,0), (n/2, n/2-3iK)}$ to be weight of the maximum weight path joining $(3iK,0)$, $(n/2, n/2-3iK)$ that is contained within a width $K$ strip around the straight line joining the end points. By definition, $G^{K}_{(3iK,0), (n/2, n/2-3iK)}$ are independent across $i$ and since $G^{K}_{(3iK,0), (n/2, n/2-3iK)}\le G_{(3iK,0), (n/2, n/2-3iK)}$ we have 
$$\mathbb{P}(G_{\mathbf{0}, (n/2 ,n/2)} \leq \mu_0 n - \epsilon n) \leq \prod_{i=0}^{\delta n} \P(G^{K}_{(3iK,0), (n/2, n/2-3iK)}\le \mu_0 n - \epsilon n).$$
One can now observe that given $\epsilon>0$, there exists $K$ sufficiently large such that $\mathbb{E} G_{\mathbf{0},(K/2,K/2)}\ge (\mu_0-\epsilon/2)K$. A large deviation estimate then implies that (see \cite[Lemma 2.2]{BGS19}) that 
$$\P\bigg(G^{K}_{\mathbf{0},(n/2,n/2)}\le (\mu_0-\frac{3\epsilon}{4})n\bigg)\le e^{-cn}.$$
By translation invariance, and choosing $\delta$ sufficiently small depending on $\epsilon,K,\mu_0$, one also gets for some $c>0$ and all $n$ sufficiently large
$$\P(G^{K}_{(3iK,0), (n/2, n/2-3iK)}\le \mu_0 n - \epsilon n)\le e^{-cn}$$
for all $i\le \delta n$. The result follows.
\end{proof}

\section{Proof of the upper bound in Theorem \ref{main_thm}}

In this section, we will show that for each $\epsilon > 0$, 
\begin{equation}\label{finite_ub}
\mathbb{P}\Big(\Mid_{0,n} \cdot \mathbf{e}_1 \geq  n/2+tn\Big) \leq n e^{-2n({\mathcal{J}}_t(\mu_0) - \epsilon)},
\end{equation}
and this implies the upper bound in Theorem \ref{main_thm} since for $n$ sufficiently large
$$
ne^{-2n({\mathcal{J}}_t(\mu_0) - \epsilon)} \leq e^{-2n({\mathcal{J}}_t(\mu_0) - 2\epsilon)}.
$$

To show \eqref{finite_ub}, it suffices to prove the following bound 
\begin{equation}\label{pt_est}
\mathbb{P}\Big(\Mid_{0,n} = (n/2+tn, n/2-tn)\Big) \leq  e^{-2n({\mathcal{J}}_t(\mu_0) - \epsilon)}.
\end{equation}
Then, \eqref{finite_ub} follows from a union bound and the monotonicity of the rate function from Proposition \ref{mono_t}, as shown below
\begin{align*}
\mathbb{P}\Big(\Mid_{0,n} \cdot \mathbf{e}_1 \geq  n/2+tn\Big) &\leq  \sum_{s = nt}^n\mathbb{P}\Big(\Mid_{0,n} = (n/2+s, n/2-s)\Big)\\
& \leq \sum_{s = \floor{nt}}^n  e^{-2n({\mathcal{J}}_{s/n}(\mu_0) - \epsilon)} \qquad \text{by \eqref{pt_est}}\\
& \leq n e^{-2n({\mathcal{J}}_t(\mu_0) - \epsilon)}.
\end{align*}

For the rest of this section, we will prove \eqref{pt_est}.
Starting with the setup, recall the shape function $
\mu_t = \lim_{n\to\infty} n^{-1} G_{(0,0), (n/2+tn, n/2-tn)} 
$. We fix $0< t \leq 1/2$ such that $\mu_0 > \mu_t$. 
For each $\epsilon >0$, we will fix a small value  $\delta >0$ such that 
\begin{equation}\label{d_fix}
\mu_0 - \delta > \mu_t \qquad \textup{ and } \qquad \;\Big|\;\mathcal{J}_t(\mu_0 - \delta) - \mathcal{J}_t(\mu_0)\;\Big|\; \leq \epsilon.
\end{equation}
This is possible due to the continuity result for the rate function in Proposition \ref{propJ}. Now observe that the midpoint $\Mid_{0,n}$ passes through a particular point if and only if the concatenation of the longest path to and from that point is also the longest path connecting $(0,0)$ to $(n,n)$, i.e.
\[
\Mid_{0,n} = (n/2 + tn, n/2 - tn) \iff G_{\mathbf{0}, (n/2+tn,n/2-tn)} + G_{(n/2+tn,n/2-tn), (n,n)} = G_{\mathbf{0}, (n,n)}
\]
By decomposing the right-hand side according to the value of $G_{\mathbf{0}, (n,n)}$ we obtain
\begin{align}
\mathbb{P}(\Mid_{0,n} = (n/2+tn, n/2-tn)) & \leq \mathbb{P}(G_{\mathbf{0}, (n,n)} \leq 2\mu_0 n - 2 \delta n) \label{up1} \\
 +  \mathbb{P}(&G_{\mathbf{0}, (n/2+tn,n/2-tn)} + G_{(n/2+tn,n/2-tn), (n,n)} \geq 2\mu_0 n - 2\delta n )\label{up2}.
\end{align}
Note that $\eqref{up1} \leq e^{-c_\delta n^2}$ due to Proposition \ref{ld_left}. Then, the term in \eqref{up2} can be estimated as follows
\begin{align*}
\eqref{up2} 
&=\mathbb{P}(G_{\mathbf{0}, (n/2+tn,n/2-tn)} + G_{(n/2+tn,n/2-tn), (n+2tn,n-2tn)} \geq 2\mu_0 n - 2\delta n )\\
& \leq   \mathbb{P}(G_{\mathbf{0}, (n+2tn,n-2tn)} \geq (\mu_0 - \delta)2n )^2\\
& \leq  e^{-2n(\mathcal{J}_t(\mu_0 - \delta)) } \qquad \textup{by Proposition \ref{ld_right}}\\
&\leq  e^{-2n(\mathcal{J}_t(\mu_0) - \epsilon)} \qquad \textup{by \eqref{d_fix}}.
\end{align*}
This completes the proof of \eqref{pt_est} and concludes this section.

\section{Proof of the lower bound in Theorem \ref{main_thm}}

In this section, we will show the weaker lower bound that for each $\epsilon > 0$, 
\begin{equation}\label{finite_lb}
\mathbb{P}\Big(\Mid_{0,n} \cdot \mathbf{e}_1 \geq  n/2+tn\Big) \geq e^{-2n({\mathcal{J}}_t(\mu_0) + \epsilon)}.
\end{equation}
The lower probability bound in Theorem \ref{main_thm} then follows from \eqref{finite_lb} above with an averaging argument, which we now explain.
Let us define the event 
$$
\mathcal{E}_k = \bigg\{ \Mid_{(k,-k),(n+k, n-k)} \cdot \mathbf{e}_1 =  n/2+k+tn \bigg\}.
$$
By translation invariance, we see that 
\begin{equation}\label{e0}
\mathbb{P}(\mathcal{E}_0) = \frac{1}{n^{2}} \mathbb{E} \bigg[ \sum_{k=0}^{n^2-1} \mathbbm{1}_{\mathcal{E}_k} \bigg] \geq \frac{1}{n^2}\mathbb{P} \bigg( \sum_{k=0}^{n^2-1} \mathbbm{1}_{\mathcal{E}_k} \geq 1 \bigg).\end{equation}

Next, we will show that
\begin{equation}\label{sum_lb}
\mathbb{P} \bigg( \sum_{k=0}^{n^2-1} \mathbbm{1}_{\mathcal{E}_k} \geq 1 \bigg) \geq e^{-2n({\mathcal{J}}_t(\mu_0) + \epsilon)}.
\end{equation}
Let us define the event 
$$
\mathcal{H}_k =  \Big\{ \Mid_{(k,-k),(n+k, n-k)} \cdot \mathbf{e}_1  \geq   n/2+k+tn \Big\}.
$$
Due to path monotonicity, the event 
$\mathcal{H}_0\cap \mathcal{H}_{2n}^c$ implies  $\sum_{k=0}^{2n} \mathbbm{1}_{\mathcal{E}_k} \geq 1$. In addition, $\mathcal{H}_0$ and $\mathcal{H}_{2n}^c$ are also independent. Thus, we have 
$$
\mathbb{P} \bigg( \sum_{k=0}^{2n} \mathbbm{1}_{\mathcal{E}_k} \geq 1 \bigg) \geq \mathbb{P}(\mathcal{H}_0\cap \mathcal{H}_{2n}^c) = \mathbb{P}(\mathcal{H}_0) \mathbb{P}( \mathcal{H}_{0}^c) \geq e^{-2n({\mathcal{J}}_t(\mu_0) + \epsilon)} (1-e^{-2n({\mathcal{J}}_t(\mu_0) - \epsilon)}). 
$$ 
Now with \eqref{e0} and \eqref{sum_lb}, it holds that 
$$\mathbb{P}(\mathcal{E}_0) \geq \frac{1}{2n^2} e^{-2n({\mathcal{J}}_t(\mu_0) + \epsilon)} \geq e^{-2n({\mathcal{J}}_t(\mu_0) + 2\epsilon)}$$
which is the lower bound in Theorem \ref{main_thm}.

Finally, we turn to the proof of \eqref{finite_lb}, and we will separate the proof into two cases when $0<t<1/2$ or $t=1/2$, given in the next two subsections.
\subsection{The case $0<t<1/2$}

To obtain the lower bound, we will plant a path with a high passage value that satisfies our desired transversal fluctuation, and then show that the geodesic will be close to the planted path. A similar construction has appeared in \cite{midpoint}, and the difference is that there it was used to study the exactly solvable model at the moderate deviation scale while here we are addressing the general model at the large deviation scale.

To start, let $\epsilon>0$ be given, by Proposition \ref{propJ}, we will fix a positive number $\delta$ such that 
\begin{equation}\label{fix_d2}
t+ 100\delta < 1/2 \quad \textup{ and } \quad \;\Big|\;\mathcal{J}_{t+100\delta}(\mu_0 + \delta) - \mathcal{J}_t(\mu_0)\;\Big|\; \leq \epsilon.
\end{equation}
In the later part of the proof, we may decrease the positive value of $\delta$ further toward zero.

\begin{figure}[t]
\begin{center}

\tikzset{every picture/.style={line width=0.75pt}} %set default line width to 0.75pt        

\begin{tikzpicture}[x=0.75pt,y=0.75pt,yscale=1.1,xscale=-1.1]

%Shape: Circle [id:dp2392312464001276] 
\draw  [fill={rgb, 255:red, 0; green, 0; blue, 0 }  ,fill opacity=1 ] (76.37,230.65) .. controls (76.37,228.63) and (78,227) .. (80.02,227) .. controls (82.03,227) and (83.67,228.63) .. (83.67,230.65) .. controls (83.67,232.67) and (82.03,234.3) .. (80.02,234.3) .. controls (78,234.3) and (76.37,232.67) .. (76.37,230.65) -- cycle ;
%Straight Lines [id:da18838110942349107] 
\draw [color={rgb, 255:red, 155; green, 155; blue, 155 }  ,draw opacity=1 ]   (98.5,86.4) -- (230.3,220.55) ;
%Shape: Circle [id:dp1147640811235322] 
\draw  [fill={rgb, 255:red, 0; green, 0; blue, 0 }  ,fill opacity=1 ] (237.03,70.32) .. controls (237.03,68.3) and (238.67,66.67) .. (240.68,66.67) .. controls (242.7,66.67) and (244.33,68.3) .. (244.33,70.32) .. controls (244.33,72.33) and (242.7,73.97) .. (240.68,73.97) .. controls (238.67,73.97) and (237.03,72.33) .. (237.03,70.32) -- cycle ;
%Straight Lines [id:da3045669479954918] 
\draw  [dash pattern={on 0.84pt off 2.51pt}]  (120.67,109.3) -- (80.02,230.65) ;
%Straight Lines [id:da17424319990309534] 
\draw  [dash pattern={on 0.84pt off 2.51pt}]  (240.68,70.32) -- (120.67,109.3) ;
%Shape: Circle [id:dp8678998404098752] 
\draw  [fill={rgb, 255:red, 0; green, 0; blue, 0 }  ,fill opacity=1 ] (85.7,201.98) .. controls (85.7,199.97) and (87.33,198.33) .. (89.35,198.33) .. controls (91.37,198.33) and (93,199.97) .. (93,201.98) .. controls (93,204) and (91.37,205.63) .. (89.35,205.63) .. controls (87.33,205.63) and (85.7,204) .. (85.7,201.98) -- cycle ;
%Shape: Circle [id:dp8338580000580653] 
\draw  [fill={rgb, 255:red, 0; green, 0; blue, 0 }  ,fill opacity=1 ] (96.69,169.98) .. controls (96.69,167.96) and (98.33,166.33) .. (100.34,166.33) .. controls (102.36,166.33) and (103.99,167.96) .. (103.99,169.98) .. controls (103.99,171.99) and (102.36,173.62) .. (100.34,173.62) .. controls (98.33,173.62) and (96.69,171.99) .. (96.69,169.98) -- cycle ;
%Shape: Circle [id:dp3454960244062313] 
\draw  [fill={rgb, 255:red, 0; green, 0; blue, 0 }  ,fill opacity=1 ] (106.41,140.83) .. controls (106.41,138.82) and (108.04,137.18) .. (110.06,137.18) .. controls (112.07,137.18) and (113.71,138.82) .. (113.71,140.83) .. controls (113.71,142.85) and (112.07,144.48) .. (110.06,144.48) .. controls (108.04,144.48) and (106.41,142.85) .. (106.41,140.83) -- cycle ;
%Shape: Circle [id:dp2286034821474674] 
\draw  [fill={rgb, 255:red, 0; green, 0; blue, 0 }  ,fill opacity=1 ] (117.02,109.3) .. controls (117.02,107.28) and (118.65,105.65) .. (120.67,105.65) .. controls (122.68,105.65) and (124.32,107.28) .. (124.32,109.3) .. controls (124.32,111.32) and (122.68,112.95) .. (120.67,112.95) .. controls (118.65,112.95) and (117.02,111.32) .. (117.02,109.3) -- cycle ;
%Shape: Circle [id:dp5968959870503454] 
\draw  [fill={rgb, 255:red, 0; green, 0; blue, 0 }  ,fill opacity=1 ] (146.73,100.44) .. controls (146.73,98.43) and (148.37,96.79) .. (150.38,96.79) .. controls (152.4,96.79) and (154.03,98.43) .. (154.03,100.44) .. controls (154.03,102.46) and (152.4,104.09) .. (150.38,104.09) .. controls (148.37,104.09) and (146.73,102.46) .. (146.73,100.44) -- cycle ;
%Shape: Circle [id:dp18243094133995008] 
\draw  [fill={rgb, 255:red, 0; green, 0; blue, 0 }  ,fill opacity=1 ] (177.03,89.81) .. controls (177.03,87.79) and (178.66,86.16) .. (180.68,86.16) .. controls (182.69,86.16) and (184.32,87.79) .. (184.32,89.81) .. controls (184.32,91.82) and (182.69,93.46) .. (180.68,93.46) .. controls (178.66,93.46) and (177.03,91.82) .. (177.03,89.81) -- cycle ;
%Shape: Circle [id:dp6044849633373095] 
\draw  [fill={rgb, 255:red, 0; green, 0; blue, 0 }  ,fill opacity=1 ] (206.45,80.67) .. controls (206.45,78.65) and (208.09,77.02) .. (210.1,77.02) .. controls (212.12,77.02) and (213.75,78.65) .. (213.75,80.67) .. controls (213.75,82.68) and (212.12,84.32) .. (210.1,84.32) .. controls (208.09,84.32) and (206.45,82.68) .. (206.45,80.67) -- cycle ;

% Text Node
\draw (71.02,239.05) node [anchor=north west][inner sep=0.75pt]    {$\mathbf h_{2J_0}  =(n,n)$};
% Text Node
\draw (251.02,62.05) node [anchor=north west][inner sep=0.75pt]    {$\mathbf h_{0} =(0,0)$};
% Text Node
\draw (111.02,117.05) node [anchor=north west][inner sep=0.75pt]    {$\mathbf h_{J_0} = (n/2+ (t+100\delta)n, n/2- (t+100\delta)n)$};

\end{tikzpicture}

\captionsetup{width=.8\linewidth}
\caption{
Planting of the vertices $\{\mathbf{h}_j\}_{j=0}^{2J_0}$. The starting point $\mathbf{h}_0$, the midpoint $\mathbf{h}_{J_0}$, and the endpoint $\mathbf{h}_{2J_0}$ are labeled in the figure above.}\label{fig3}
\end{center}
\end{figure}
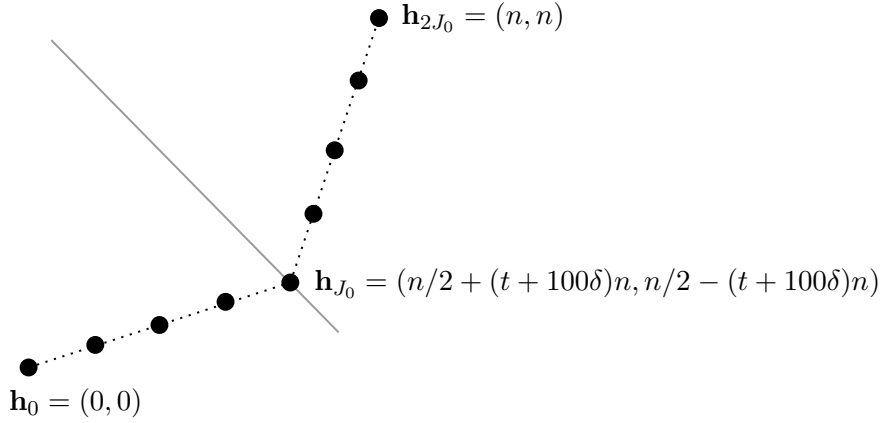
Next, as illustrated in Figure \ref{fig3}, we will look at a sequence of points whose $\ell_1$-distances are $2\delta^5n$ apart along the two straight line segments in between $\mathbf{0}$,  $(n/2+ (t+100\delta)n, n/2- (t+100\delta)n)$ and $(n,n)$. We may further decrease the value of $\delta$ so that $\delta^{-5}$ is an integer. If we denote $\delta^{-5}$ as $J_0$, because the $\ell_1$-distance between $\mathbf{0}$ and $(n/2+ (t+100\delta)n, n/2- (t+100\delta)n)$ is exactly $n$,  we can label the points as $\mathbf{h}_j$ for $j=0, \dots, 2J_0$, where $\mathbf{h}_0 = \mathbf{0}$, $\mathbf{h}_{J_0} =  (n/2+ (t+100\delta)n, n/2- (t+100\delta)n)$ and $\mathbf{h}_{2J_0} = (n,n)$.

Define the event 
$$
\mathcal{A} = \bigcap_{i=1}^{2J_0} \mathcal{A}_i = \bigcap_{i=1}^{2J_0} \Big\{G_{\mathbf h_{i-1}, \mathbf h_i} \geq (\mu_0+\delta)2\delta^5{n}\Big\}.
$$
By Proposition \ref{ld_right} we have
\[
\mathbb{P}(\mathcal{A}_i) \geq e^{-2\delta^5n(\mathcal{J}_{t+100 \delta} +\epsilon)}
\]
for each $i$. Then by independence and \eqref{fix_d2} it holds that
$$\mathbb{P}(\mathcal{A}) \geq e^{-2n(\mathcal{J}_{t+100 \delta}(\mu_0 + \delta)+\epsilon) } \geq e^{-2n(\mathcal{J}_{t}(\mu_0) + 2\epsilon)}.$$

\begin{figure}[t]
\begin{center}
\tikzset{every picture/.style={line width=0.75pt}} %set default line width to 0.75pt        
\begin{tikzpicture}[x=0.75pt,y=0.75pt,yscale=-1.3,xscale=1.3]

%Curve Lines [id:da23928717918199727] 
\draw [color={rgb, 255:red, 208; green, 2; blue, 27 }  ,draw opacity=1 ][line width=1.5]    (115.67,225.53) .. controls (160.6,220.5) and (116.93,198.83) .. (152.13,185.1) ;
%Curve Lines [id:da10156648576827165] 
\draw [color={rgb, 255:red, 208; green, 2; blue, 27 }  ,draw opacity=1 ][line width=1.5]    (155.23,185.1) .. controls (214.93,172.77) and (188.27,97.5) .. (245.23,95.43) ;
%Curve Lines [id:da552842695651975] 
\draw [color={rgb, 255:red, 208; green, 2; blue, 27 }  ,draw opacity=1 ][line width=1.5]    (245.23,95.43) .. controls (276.6,92.37) and (261.6,77.7) .. (274.57,65.43) ;
%Curve Lines [id:da7338863858306455] 
\draw [color={rgb, 255:red, 74; green, 144; blue, 226 }  ,draw opacity=0.66 ][line width=1.5]    (169.07,198.43) .. controls (211.4,151.77) and (255.73,183.1) .. (259.07,109.1) ;
%Straight Lines [id:da5872696970341162] 
\draw [color={rgb, 255:red, 155; green, 155; blue, 155 }  ,draw opacity=0.55 ]   (102.13,135.1) -- (202.13,235.1) ;
%Straight Lines [id:da9510815807621278] 
\draw [color={rgb, 255:red, 155; green, 155; blue, 155 }  ,draw opacity=0.55 ]   (195.23,45.43) -- (295.23,145.43) ;
%Curve Lines [id:da8585415855670777] 
\draw [color={rgb, 255:red, 74; green, 144; blue, 226 }  ,draw opacity=0.66 ][line width=1.5]    (140.4,206.9) .. controls (148.07,201.57) and (154.07,226.5) .. (169.07,198.43) ;
%Curve Lines [id:da6990290884714354] 
\draw [color={rgb, 255:red, 74; green, 144; blue, 226 }  ,draw opacity=0.66 ][line width=1.5]    (259.07,109.1) .. controls (259.07,98.23) and (269.07,104.9) .. (265.07,90.9) ;
%Curve Lines [id:da9246273848831208] 
\draw [color={rgb, 255:red, 74; green, 144; blue, 226 }  ,draw opacity=0.66 ][line width=1.5]    (88.4,224.57) .. controls (101.07,205.23) and (118.73,219.57) .. (140.4,206.9) ;
%Curve Lines [id:da3763903610475099] 
\draw [color={rgb, 255:red, 74; green, 144; blue, 226 }  ,draw opacity=0.66 ][line width=1.5]    (265.07,90.9) .. controls (265.07,80.03) and (275.4,73.5) .. (309.4,73.17) ;
%Curve Lines [id:da1325988111405907] 
\draw [color={rgb, 255:red, 248; green, 231; blue, 28 }  ,draw opacity=1 ][line width=3]  [dash pattern={on 3.38pt off 3.27pt}]  (169.07,198.43) .. controls (208.13,152.9) and (255.73,183.1) .. (259.07,109.1) ;
%Curve Lines [id:da7333809978642245] 
\draw [color={rgb, 255:red, 245; green, 166; blue, 35 }  ,draw opacity=1 ][line width=3]  [dash pattern={on 3.38pt off 3.27pt}]  (139.33,207.23) .. controls (147,201.9) and (152.4,229.23) .. (167.4,201.17) ;
%Curve Lines [id:da790689982542518] 
\draw [color={rgb, 255:red, 245; green, 166; blue, 35 }  ,draw opacity=1 ][line width=3]  [dash pattern={on 3.38pt off 3.27pt}]  (260.07,106.23) .. controls (260.07,95.37) and (268.13,106.57) .. (265.07,90.9) ;
%Curve Lines [id:da32145765978117335] 
\draw [color={rgb, 255:red, 245; green, 166; blue, 35 }  ,draw opacity=1 ][line width=3]  [dash pattern={on 3.38pt off 3.27pt}]  (115.67,225.53) .. controls (138.73,221.57) and (137.07,219.9) .. (137.67,208) ;
%Curve Lines [id:da0032691288619368652] 
\draw [color={rgb, 255:red, 245; green, 166; blue, 35 }  ,draw opacity=1 ][line width=3]  [dash pattern={on 3.38pt off 3.27pt}]  (274.57,65.43) .. controls (267.4,73.57) and (266.73,85.57) .. (265.07,90.9) ;
%Shape: Circle [id:dp8494269162527333] 
\draw  [fill={rgb, 255:red, 0; green, 0; blue, 0 }  ,fill opacity=1 ] (112.57,225.53) .. controls (112.57,223.82) and (113.95,222.43) .. (115.67,222.43) .. controls (117.38,222.43) and (118.77,223.82) .. (118.77,225.53) .. controls (118.77,227.25) and (117.38,228.63) .. (115.67,228.63) .. controls (113.95,228.63) and (112.57,227.25) .. (112.57,225.53) -- cycle ;
%Shape: Square [id:dp517940969132503] 
\draw  [fill={rgb, 255:red, 0; green, 0; blue, 0 }  ,fill opacity=1 ] (135.83,207.23) -- (139.33,207.23) -- (139.33,210.73) -- (135.83,210.73) -- cycle ;
%Shape: Square [id:dp49190063136127504] 
\draw  [fill={rgb, 255:red, 0; green, 0; blue, 0 }  ,fill opacity=1 ] (263.5,86.23) -- (267,86.23) -- (267,89.73) -- (263.5,89.73) -- cycle ;
%Shape: Circle [id:dp23578764367954475] 
\draw  [fill={rgb, 255:red, 0; green, 0; blue, 0 }  ,fill opacity=1 ] (149.03,185.1) .. controls (149.03,183.39) and (150.42,182) .. (152.13,182) .. controls (153.85,182) and (155.23,183.39) .. (155.23,185.1) .. controls (155.23,186.81) and (153.85,188.2) .. (152.13,188.2) .. controls (150.42,188.2) and (149.03,186.81) .. (149.03,185.1) -- cycle ;
%Shape: Circle [id:dp9884812412801904] 
\draw  [fill={rgb, 255:red, 0; green, 0; blue, 0 }  ,fill opacity=1 ] (242.13,95.43) .. controls (242.13,93.72) and (243.52,92.33) .. (245.23,92.33) .. controls (246.95,92.33) and (248.33,93.72) .. (248.33,95.43) .. controls (248.33,97.15) and (246.95,98.53) .. (245.23,98.53) .. controls (243.52,98.53) and (242.13,97.15) .. (242.13,95.43) -- cycle ;
%Shape: Circle [id:dp6629723290168177] 
\draw  [fill={rgb, 255:red, 0; green, 0; blue, 0 }  ,fill opacity=1 ] (271.47,65.43) .. controls (271.47,63.72) and (272.85,62.33) .. (274.57,62.33) .. controls (276.28,62.33) and (277.67,63.72) .. (277.67,65.43) .. controls (277.67,67.15) and (276.28,68.53) .. (274.57,68.53) .. controls (272.85,68.53) and (271.47,67.15) .. (271.47,65.43) -- cycle ;

% Text Node
\draw (103.33,230.8) node [anchor=north west][inner sep=0.75pt]    {$\mathbf{h}_{m^*-1}$};
% Text Node
\draw (148.33,168.13) node [anchor=north west][inner sep=0.75pt]    {$\mathbf{h}_{m^*}$};
% Text Node
\draw (234,100.47) node [anchor=north west][inner sep=0.75pt]    {$\mathbf{h}_{k^*}$};
% Text Node
\draw (278,49.47) node [anchor=north west][inner sep=0.75pt]    {$\mathbf{h}_{k^*+1}$};
% Text Node
\draw (124,194.6) node [anchor=north west][inner sep=0.75pt]    {$\mathbf a^*$};
% Text Node
\draw (270.67,82.93) node [anchor=north west][inner sep=0.75pt]    {$\mathbf b^*$};
% Text Node
\draw (227.67,165.93) node [anchor=north west][inner sep=0.75pt]    {$\gamma_{\textup{mid}}^{*}$};
% Text Node
\draw (242.67,65) node [anchor=north west][inner sep=0.75pt]    {$\gamma_{\textup{left}}^{*}$};
% Text Node
\draw (146.67,217.63) node [anchor=north west][inner sep=0.75pt]    {$\gamma_{\textup{right}}^{*}$};
% Text Node
\draw (312.67,68.4) node [anchor=north west][inner sep=0.75pt]    {$\Gamma_{0,n}$};
% Text Node
\draw (186.67,127.07) node [anchor=north west][inner sep=0.75pt]    {$\Theta$};

\end{tikzpicture}
\captionsetup{width=.8\linewidth}
\caption{
An illustration of the definitions: The concatenation of geodesics connecting the points $\mathbf{h}_i$ is depicted in red and is denoted by $\Theta$. The geodesic from $(0,0)$ to $(n,n)$ is represented in blue. From these two paths, we construct a new path, $\gamma^*$, starting at $\mathbf{h}_{m^*-1}$ and ending at $\mathbf{h}_{k^*+1}$, shown as a yellow dotted line. Subsequently, the path $\gamma^*$ is further divided into three disjoint segments.}\label{fig1}
\end{center}
\end{figure}
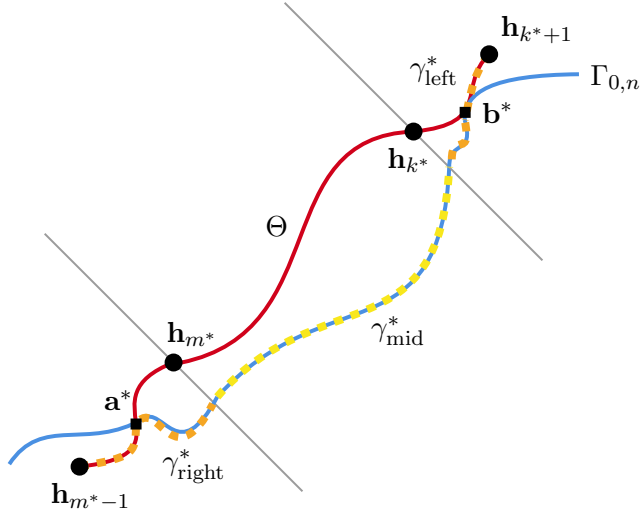

Next, we will show that $\Mid_{0,n}$ is to the right of $(n/2 + tn, n/2-tn)$ with probability at least $1/10$ when conditioned on the event $\mathcal{A}$ above. To do this, we will start with a coalescence bound between the geodesic $\Gamma_{0,n}$ and the concatenation of the geodesics between $\mathbf h_{i-1}$ and $\mathbf h_i$, which is defined below
$$\Theta = \textup{the concatenation of the geodesics $\Gamma_{\mathbf{h}_{i-1}, \mathbf{h}_i}$ for $i=1, \dots, 2J_0$}.$$

The following definitions are illustrated in Figure \ref{fig1}.
Let us denote the last intersection point between $\Gamma_{0,n}$ and $\Theta$ below $\mathcal{L}_{n/2}$ as $\mathbf{a}^*$, and let $\mathbf{b}^*$ denote the first intersection point $\Gamma_{0,n}$ and $\Theta$ above $\mathcal{L}_{n/2}$. Define the event 
$$\mathcal{D} = \Big\{\textup{at least one of  $\mathbf{a}^*$ or $\mathbf{b}^*$ is between $\mathcal{L}_{n/2-\delta n}$ and $\mathcal{L}_{n/2+\delta n}$}\Big\}.$$
and we will show that, $\mathbb{P}(\mathcal{D}|\mathcal{A}) \geq 1/10$.

Note that since $\mathbf{a}^*$ and $\mathbf{b}^*$ lie on the path $\Theta$, which passes through the point $\mathbf{h}_{J_0} = (n/2 + (t + 100\delta)n, n/2 - (t + 100\delta)n)$, when event $\mathcal{D}$ occurs, at least one of the points $\mathbf{a}^*$ or $\mathbf{b}^*$ must be contained within the ball $\mathbf{S}$ defined as follows
$$\mathbf{S} = \left\{\mathbf{x} : |\mathbf{x} - \mathbf{h}_{J_0}|_2 \leq 10 \delta n\right\}.$$
Furthermore, since the midpoint $\text{Mid}_{0,n}$ lies on the segment of $\Gamma_{0,n}$ between $\mathbf{a}^*$ and $\mathbf{b}^*$ and is also situated on the line $\mathcal{L}_{n/2}$, if either $\mathbf{a}^*$ or $\mathbf{b}^*$ is contained within $\mathbf{S}$, then $\text{Mid}_{0,n}$ must also be included in $\mathbf{S}$. Now, since the point $\bigg(n/2 + tn, n/2 - tn\bigg)$ is positioned to the left of the ball $\mathbf{S}$, we can conclude that $\text{Mid}_{0,n}$ must be to the right of the point $\bigg(n/2 + tn, n/2 - tn\bigg)$.

Let $m^*$ denote the smallest index such that $\mathbf{a}^*$ is below $\mathcal{L}_{\mathbf{h}_{m^*}}$, and similarly let $k^*$ be the largest index such that $\mathbf{b}^*$ is above $\mathcal{L}_{\mathbf{h}_{k^*}}$.
Let us define the event 
$$\mathcal{C}_{m,k} =  \{m^* = m \textup{ and }k^* = k\}.$$
To lower bound $\mathbb{P}(\mathcal{D}|\mathcal{A})$, 
it suffices for us to upper bound its complement  
\begin{equation}\label{unionC}
\sum_{m = 1}^{J_0-\delta^{-4}} \sum_{k = J_0 + \delta^{-4}}^{2J_0}\mathbb{P}( \mathcal{C}_{m, j}| \mathcal{A}).
\end{equation}

First, we will try to bound $\mathbb{P}(\mathcal{C}_{m, k}| \mathcal{A}).$
To do this, let us define a path $\gamma^*$ between $\mathbf{h}_{m-1}$ and $\mathbf{h}_{k+1}$, obtained by following the segment of the geodesic $\Gamma_{\mathbf{h}_{m-1}, \mathbf{h}_{m}}$ from $\mathbf{h}_{m-1}$ to $\mathbf{a}^*$, then following the segment of the geodesic $\Gamma_{0, n}$ from $\mathbf{a}^*$ to $\mathbf{b}^*$, and lastly following the geodesic $\Gamma_{\mathbf{h}_{k}, \mathbf{h}_{k+1}}$ from $\mathbf{b}^*$ to $\mathbf{h}_{k+1}$. In particular, note that the last-passage value of $\gamma^*$ must be bigger than the last-passage value along $\Theta$ between $\mathbf{h}_{m-1}$ and $\mathbf{h}_{k+1}$. Thus on the event $\mathcal{C}_{m,k}\cap \mathcal{A}$, it holds that 
$$G(\gamma^*) \geq (\mu_0 + \delta)(k+1-(m-1))2\delta^5n.$$
 
%$$G(\gamma^*) \leq G_{\mathbf{h}_{m-1}, \mathcal{L}_{\mathbf{h}_{m}}^{\delta^3 n}} + G_{\mathcal{L}_{\mathbf{h}_{m}}^{\delta^3 n}, \mathcal{L}_{\mathbf{h}_{k}}^{\delta^3 n}} + G_{\mathcal{L}_{\mathbf{h}_{m}}^{\delta^3 n}, \mathbf{h}_{k+1}}.$$
Now, let us separate the path $\gamma^*$ into three pieces $\gamma^*_{\textup{left}}, \gamma^*_{\textup{mid}}$ and $\gamma^*_{\textup{right}}$, which are obtained by cutting $\gamma^*$ using the two antidiagonal lines $\mathcal{L}_{\mathbf{h}_{m}}$ and $\mathcal{L}_{\mathbf{h}_k}$. Then, it holds that 
\begin{align}
\mathbb{P}(\mathcal{C}_{m, k}| \mathcal{A})
& \leq \mathbb{P}\Big(G(\gamma^*) \geq (\mu_0 + \delta)(k+1-(m-1))2\delta^5n \;\Big|\; \mathcal{A}\Big)\nonumber \\
& \leq \mathbb{P}\Big(G(\gamma^*_{\textup{left}}) \geq \mu_02\delta^5n + \frac{\delta(k+1-(m-1))}{3}2\delta^5n \;\Big|\;\mathcal{A}\Big)\label{left}\\
& \qquad + \mathbb{P}\Big(G(\gamma^*_{\textup{mid}}) \geq \mu_0(k-m)2\delta^5n + \frac{\delta(k+1-(m-1))}{3}2\delta^5n  \;\Big|\;\mathcal{A}\Big)\label{mid}\\
& \qquad \qquad + \mathbb{P}\Big(G(\gamma^*_{\textup{right}}) \geq \mu_02\delta^5n + \frac{\delta(k+1-(m-1))}{3}2\delta^5n\;\Big|\;\mathcal{A}\Big)\label{right}
\end{align}

We will start with the estimate for \eqref{mid}. Let $\mathcal{L}_\mathbf{a}^b$ denote the antidiagonal segment of $\ell_1$-length $b$ with midpoint at $\mathbf{a}$. By the definition of $\gamma^*_{\textup{mid}}$, it must be disjoint from the random path $\Theta$, and it must start within $\mathcal{L}_{\mathbf h_m}^{2n}$ and end within $\mathcal{L}_{\mathbf h_k}^{2n}$. Let us define the event 
$$\mathcal{M} = \Big\{\exists \textup{ $\gamma$ from }\mathcal{L}_{\mathbf h_m}^{2n} \textup{ to } \mathcal{L}_{\mathbf h_k}^{2 n} \textup{ such that }G(\gamma)  \geq \mu_0(k-m)2\delta^5n + \frac{\delta(k+1-(m-1))}{3}2\delta^5n \Big\}.$$
Let $A\Box B$ denote the disjoint occurrence of the events $A$ and $B$. Since $\mathcal{M}$ and $\mathcal{A}$ are both increasing,  by the BKR inequality \cite{BKR_ineq}
$$\mathbb{P}\Big(\Big\{G(\gamma^*_{\textup{mid}}) \geq \mu_0(k-m)2\delta^3n + \frac{\delta(k+1-(m-1))}{3}2\delta^5n\Big\}\cap \mathcal{A} \Big) \leq \mathbb{P}(\mathcal{M} \Box \mathcal{A}) \leq \mathbb{P}(\mathcal{M}) \mathbb{P}(\mathcal{A}).$$
Therefore, by a union bound,
\begin{align*}
\eqref{mid} &\leq \mathbb{P}(\mathcal{M}) \\
&\leq \sum_{\mathbf{p}\in \mathcal{L}_{\mathbf h_m}^{2 n}, \mathbf{q} \in \mathcal{L}_{\mathbf h_m}^{2n}} \mathbb{P}\Big(G_{\mathbf p, \mathbf q} \geq \mu_0(k-m)2\delta^5n + \frac{\delta(k+1-(m-1))}{3}2\delta^5n  \Big)\\
&\leq 10n^2 e^{-c_\delta n} \quad \textup{ by Proposition \ref{ld_right}}.
\end{align*}

Now, turn our attention to \eqref{left}, and the estimate for \eqref{right} follows from the same argument. It holds that 
\begin{align}
\eqref{left} &\leq \mathbb{P}\Big(G_{\mathbf{h}_m, \mathcal{L}_{\mathbf{h}_{m+1}}} \geq \mu_02\delta^5n + \frac{\delta(k+1-(m-1))}{3}2\delta^5n \;\Big|\;\mathcal{A}\Big)\nonumber\\
&= \mathbb{P}\Big(G_{\mathbf{h}_m, \mathcal{L}_{\mathbf{h}_{m+1}}} \geq \mu_02\delta^5n + \frac{\delta(k+1-(m-1))}{3}2\delta^5n \;\Big|\;\mathcal{A}_{m+1}\Big)\nonumber\\
& \leq \frac{\mathbb{P}\Big(G_{\mathbf{h}_m, \mathcal{L}_{\mathbf{h}_{m+1}}} \geq \mu_02\delta^5n + \frac{\delta(k+1-(m-1))}{3}2\delta^5n\Big)}{\mathbb{P}\Big(G_{\mathbf{h}_m,\mathbf{h}_{m+1}} \geq \mu_02\delta^5{n} + 2\delta^6{n}\Big)}. \label{ratio}
\end{align}
By Proposition \ref{ld_right}, the denominator in \eqref{ratio} is lower bounded as 
$$\mathbb{P}\Big(G_{\mathbf{h}_m,\mathbf{h}_{m+1}} \geq \mu_02\delta^5{n} + 2\delta^6{n}\Big) \geq e^{-2\delta^5n(\mathcal{J}_{t+100\delta}(\mu_0 + \delta) + \epsilon)}.$$ While the numerator in \eqref{ratio} is upper bounded by 
\begin{align*}\mathbb{P}\Big(G_{\mathbf{h}_m, \mathcal{L}_{\mathbf{h}_{m+1}}} \geq \mu_02\delta^5n + \frac{\delta(k+1-(m-1))}{3}2\delta^5n\Big) &\leq ne^{-2\delta^5n\mathcal{J}_{t+100\delta}(\mu_0 + \frac{\delta(k+1-(m-1))}{3})} \\
&\leq ne^{-2\delta^5n\mathcal{J}_{t+100\delta}(\mu_0 + \delta^{-1})}.
\end{align*}
Because $\mathcal{J}_{t+100\delta}(r)$ is non-negative, convex, and strictly increasing when $r \geq \mu_{t+100\delta}$, for $\delta>0$ sufficiently small, we must have 
$\eqref{ratio} \leq ne^{-c_\delta n}.$

Now, we have shown that $\mathbb{P}(\mathcal{C}_{m,k}|\mathcal{A}) \leq 100n^2 e^{-c_\delta n}.$ This implies that $\eqref{unionC} \leq \delta^{-100}n^2 e^{-c_\delta n}$, which can be made less than $1/2$ for all large $n$ depending on $\delta$. Thus, we have shown 
$$
\mathbb{P}\Big(\Mid_{0,n} \cdot \mathbf{e}_1 \geq  n/2+tn\Big)
\geq \mathbb{P}(\mathcal{D})
\geq \mathbb{P}(\mathcal{D}|\mathcal{A}) \mathbb{P}(\mathcal{A})
 \geq \tfrac12 e^{-2n(\mathcal{J}_{t}(\mu_0) + \epsilon)} \geq e^{-2n(\mathcal{J}_{t}(\mu_0) + 2\epsilon)}
$$
provided $n$ is sufficiently large. With this, we conclude this subsection.

\subsection{The case $t=1/2$}

For $t < 1/2$ the last section bounds the probability of the midpoint being to the right of a point on the anti-diagonal line $\mathcal{L}_n$. For $t = 1/2$, there are no points to the right of $(n,0)$ that lie on this anti-diagonal line and so the argument does not immediately apply. In this section, we give a modified argument for the $t = 1/2$ case.

For simplicity of the notation, let $I = \mathcal{J}_{1/2}$.
Let $0<\alpha< \beta < 1/2$ be two constants which we will fix later. Define $a =  n^\alpha$ and $b=n^\beta$. By the large deviation lower bound in Proposition \ref{ld_right} and the continuity result in Proposition \ref{propJ}, for each fixed $\epsilon > 0$, we choose $\delta > 0$ such that 
$$\mathbb{P}(G_{\mathbf{0}, (a,0)} > a(\mu_0 + \delta)) \geq  e^{-a(I(\mu_0 +\delta) + \epsilon/2)} \geq e^{-a(I(\mu_0) + \epsilon)}.$$

Next, we will define three events $\mathcal{A}$, $\mathcal{B}$, $\mathcal{C}$ and show that when conditioned on their intersection, the probability that geodesic being the corner path is at least 1/10. 
To start, let $\mathcal{A}$ denote the event that 
$$
\mathcal{A} = \bigcap_{i, j = 0}^{n/a-1}\Big\{ G_{(ia,0), (ia+a, 0)} \geq a(\mu_0 +\delta) \text{ and } G_{(n, ja), (n, ja+a)} \geq a(\mu_0 +\delta)\Big\}.
$$
Next, fix $0<c_1<c_2$ be such that $\mathbb{P}(\omega_\mathbf{0}<c_1) > 0$ and $\mathbb{P}(\omega_\mathbf{0}>c_2) > 0$. Define the event $\mathcal{B}$ as 
$$\mathcal{B} = \Big\{\omega_{(i,j)} < c_1 \textup{ for } n-b \leq i < n \textup{ and } 0<j\leq b \Big\}.$$
In addition, let $\mathcal{C}$ denote the event that 
$$\mathcal{C} = \Big\{\omega_{(i,0)} > c_2 \textup{ for }n-b\leq i \leq n\Big\} \bigcap \Big\{\omega_{(n,j)} > c_2 \textup{ for } 0\leq j \leq b \Big\}. $$

Next, we will lower bound the probability $\mathbb{P}(\mathcal{A} \cap \mathcal{B}\cap \mathcal{C})$. Since $b = n^{\beta}$ where $\beta<  1/2$, $\mathbb{P}(\mathcal{B})$ and $\mathbb{P}(\mathcal{C})$ are both lower bounded by $e^{-\epsilon n}$ when $n$ is large. 
Note that $\mathcal{B}$ is independent of $\mathcal{A}$ and $\mathcal{C}$, while $\mathcal{A}$ and $\mathcal{C}$ are both increasing. Hence, by Proposition \ref{ld_right} and the FKG inequality, it holds that 
$$\mathbb{P}(\mathcal{A} \cap \mathcal{B} \cap \mathcal{C}) \geq \mathbb{P}(\mathcal{A}) \mathbb{P} (\mathcal{B})\mathbb{P}(\mathcal{C}) \geq e^{-2n(I(\mu_0) + 3\epsilon)}.$$

For the rest of this section, we will show the following lemma, which will finish the proof of our desired lower-bound 
$$\mathbb{P}(\Mid_{0,n} = (n,0)) \geq \frac{1}{100}e^{-2n(I(\mu_0) + 3\epsilon)}.$$
\begin{lemma} It holds that 
$$\mathbb{P}(\Mid_{0,n} = (n,0) | \mathcal{A}\cap \mathcal{B}\cap\mathcal{C}) \geq 1/2.$$
\end{lemma}

\begin{proof}
For integers $k$ and $\ell$ such that $0\leq k, \ell \leq n$, define $\mathcal{D}_{k, \ell}$ as the following event 
$$\mathcal{D}_{k, \ell} = \Big\{\Gamma_{0,n} \textup{ exits the $\mathbf{e}_1$-axis at $(k,0)$ and first intersect the line $x=n$ at $(n,\ell)$}\Big\}.$$
To prove the lower bound in the statement of our lemma, we shall upper bound its complement 
\begin{equation}\label{sum_D}
\sum_{k, \ell=0}^n  \mathbb{P}(\mathcal{D}_{k, \ell}|\mathcal{A}) < 1/2.
\end{equation}

First, note that if $k\geq n-b$  and $\ell \leq b$, by the definition of the events $\mathcal{B}$ and $\mathcal{C}$, the passage value of any paths between $(k,0)$ and $(n,\ell)$ inside the bulk will be strictly less than the path $(k,0) \shortrightarrow (n,0) \shortrightarrow (n,l)$, thus
$\mathbb{P}(\mathcal{D}_{k, \ell}|\mathcal{A}\cap \mathcal{B} \cap \mathcal{C}) = 0.$

Now, suppose that $k<n-b$ or $\ell > b$, then on the event $\mathcal{D}_{k, \ell}$ and $\mathcal{A}\cap \mathcal{B} \cap\mathcal{C}$, it must hold that
$$G_{(k,1), (n-1, \ell)} \geq ((n-k-1) + (\ell-1) - 2a)(\mu_0 + \delta).$$
Also note that the above event is independent of $\mathcal{A}$, $\mathcal{B}$ and $\mathcal{C}$, we have the following bounds
\begin{align}
\mathbb{P}(\mathcal{D}_{k, \ell}|\mathcal{A}\cap \mathcal{B} \cap \mathcal{C}) \nonumber
&\leq \mathbb{P}(G_{(k,1), (n-1, \ell)} \geq ((n-k-1) + (\ell-1) - 2a)(\mu_0 + \delta)|\mathcal{A}\cap \mathcal{B} \cap \mathcal{C}) \nonumber\\
&=\mathbb{P}(G_{(k,1), (n-1, \ell)} \geq ((n-k-1) + (\ell-1) - 2a)(\mu_0 + \delta)).\label{kl_bd}
\end{align}
To continue, fix $n-k$ and $\ell$, let us define the time constant 
$$\mu_* = \lim_{m\to \infty} \frac{1}{m} G_{m(k,1), m(n-1, \ell)}.$$
Following from superadditivity, it holds that 
$\mu_* \leq ((n-k-1) + (\ell-1)) \mu_0.$
Since $\max\{(n-k), \ell\} \geq b$ and $b$ is much larger than $a$, for some constant $c>0$, it holds that 
$$\Big((n-k-1) + (\ell-1) - 2a\Big)\Big(\mu_0 + \delta\Big) \geq \mu^* + c(n-k+\ell).$$
Then, by Proposition \ref{ld_right}, it holds that 
\begin{align*}
\eqref{kl_bd}
&\leq \mathbb{P}(G_{(k,1), (n-1, \ell)} \geq \mu^* + c(n-k+\ell)) \\
& \leq e^{-c'(n-k+\ell)}\leq e^{-c'n^{\alpha/10}}.
\end{align*}
With this, we have shown \eqref{sum_D}
$$\sum_{k, \ell=0}^n  \mathbb{P}(\mathcal{D}_{k, \ell}|\mathcal{A}) \leq n^2 e^{-c'n^{\alpha/10}} < 1/2 $$
and hence finished the proof of this lemma. 
\end{proof}

\bibliographystyle{amsplain}
\bibliography{refs}

\providecommand{\bysame}{\leavevmode\hbox to3em{\hrulefill}\thinspace}
\providecommand{\MR}{\relax\ifhmode\unskip\space\fi MR }
\providecommand{\MRhref}[2]{%
  \href{http://www.ams.org/mathscinet-getitem?mr=#1}{#2}
}
\providecommand{\href}[2]{#2}
\begin{thebibliography}{10}

\bibitem{aga-23}
Pranay Agarwal, \emph{Lower bound for large local transversal fluctuations of
  geodesics in last passage percolation}, 2023, {\tt arXiv:2311.00360}.

\bibitem{midpoint}
Pranay Agarwal and Riddhipratim Basu, \emph{Sharp deviation bounds for midpoint
  and endpoint of geodesics in exponential last passage percolation}, 2024,
  {\tt arXiv:2405.18056}.

\bibitem{Alb-Cat-21}
Tom Alberts and Eric Cator, \emph{On the passage time geometry of the last
  passage percolation problem}, ALEA Lat. Am. J. Probab. Math. Stat.
  \textbf{18} (2021), no.~1, 211--247. \MR{4198875}

\bibitem{BKR_ineq}
Richard Arratia, Skip Garibaldi, and Alfred~W. Hales, \emph{The van den
  {B}erg--{K}esten-{R}eimer operator and inequality for infinite spaces},
  Bernoulli \textbf{24} (2018), no.~1, 433--448. \MR{3706764}

\bibitem{MR3189069}
Antonio Auffinger and Michael Damron, \emph{A simplified proof of the relation
  between scaling exponents in first-passage percolation}, Ann. Probab.
  \textbf{42} (2014), no.~3, 1197--1211. \MR{3189069}

\bibitem{Auf-Dam-Han-17}
Antonio Auffinger, Jack Hanson, and Michael Damron, \emph{50 years of first
  passage percolation}, University Lecture Series, vol.~68, American
  Mathematical Society, Providence, RI, 2017.

\bibitem{BGS19}
Riddhipratim Basu, Shirshendu Ganguly, and Allan Sly, \emph{Delocalization of
  polymers in lower tail large deviation}, Comm. Math. Phys. \textbf{370}
  (2019), no.~3, 781--806. \MR{3995919}

\bibitem{timecorrflat}
Riddhipratim Basu, Shirshendu Ganguly, and Lingfu Zhang, \emph{Temporal
  correlation in last passage percolation with flat initial condition via
  {B}rownian comparison}, Comm. Math. Phys. \textbf{383} (2021), no.~3,
  1805--1888. \MR{4244262}

\bibitem{slowbondproblem}
Riddhipratim Basu, Vladas Sidoravicius, and Allan Sly, \emph{Last passage
  percolation with a defect line and the solution of the slow bond problem},
  2014, {\tt arXiv:1408.3464}.

\bibitem{MR3010809}
Sourav Chatterjee, \emph{The universal relation between scaling exponents in
  first-passage percolation}, Ann. of Math. (2) \textbf{177} (2013), no.~2,
  663--697. \MR{3010809}

\bibitem{dl_ldp}
Sayan Das, Duncan Dauvergne, and Bálint Virág, \emph{Upper tail large
  deviations of the directed landscape}, 2024, {\tt arXiv:2405.14924}.

\bibitem{KPZ_DL}
Duncan Dauvergne, Janosch Ortmann, and B\'{a}lint Vir\'{a}g, \emph{The directed
  landscape}, Acta Math. \textbf{229} (2022), no.~2, 201--285. \MR{4554223}

\bibitem{Dur-96}
Richard Durrett, \emph{Probability: theory and examples}, second ed., Duxbury
  Press, Belmont, CA, 1996. \MR{1609153 (98m:60001)}

\bibitem{LDP_poly}
Nicos Georgiou and Timo Sepp\"{a}l\"{a}inen, \emph{Large deviation rate
  functions for the partition function in a log-gamma distributed random
  potential}, Ann. Probab. \textbf{41} (2013), no.~6, 4248--4286. \MR{3161474}

\bibitem{ham-sar-20}
Alan Hammond and Sourav Sarkar, \emph{Modulus of continuity for polymer
  fluctuations and weight profiles in {P}oissonian last passage percolation},
  Electron. J. Probab. \textbf{25} (2020), Paper No. 29, 38. \MR{4073690}

\bibitem{Kar-Par-Zha-86}
Mehran Kardar, Giorgio Parisi, and Yi-Cheng Zhang, \emph{Dynamic scaling of
  growing interfaces}, Phys. Rev. Lett. \textbf{56} (1986), 889--892.

\bibitem{Kes-86-stflour}
Harry Kesten, \emph{Aspects of first passage percolation}, \'{E}cole
  d'{\'e}t{\'e} de probabilit{\'e}s de {S}aint-{F}lour, {XIV}---1984, Lecture
  Notes in Math., vol. 1180, Springer, Berlin, 1986, pp.~125--264. \MR{876084}

\bibitem{Liu-22a}
Zhipeng Liu, \emph{One-point distribution of the geodesic in directed last
  passage percolation}, Probab. Theory Related Fields \textbf{184} (2022),
  no.~1-2, 425--491. \MR{4498515}

\bibitem{Mar-04}
James~B. Martin, \emph{Limiting shape for directed percolation models}, Ann.
  Probab. \textbf{32} (2004), no.~4, 2908--2937. \MR{2094434}

\bibitem{transfluc}
Charles~M. Newman and Marcelo S.~T. Piza, \emph{Divergence of shape
  fluctuations in two dimensions}, Ann. Probab. \textbf{23} (1995), no.~3,
  977--1005. \MR{1349159}

\bibitem{Ros-81}
H.~Rost, \emph{Nonequilibrium behaviour of a many particle process: density
  profile and local equilibria}, Z. Wahrsch. Verw. Gebiete \textbf{58} (1981),
  no.~1, 41--53. \MR{635270}

\bibitem{lb_tf}
Xiao Shen, \emph{Lower bound for large transversal fluctuations in exactly
  solvable {KPZ} models}, 2024, {\tt arXiv:2402.16332}.

\bibitem{var_J_fpp}
Julien Verges, \emph{Large deviation principle at speed $n$ for the random
  metric in first-passage percolation}, 2024, {\tt arXiv:2412.03320}.

\end{thebibliography}

\end{document}